\documentclass[11pt]{article}

\usepackage{enumerate}
\usepackage{amsmath}
\usepackage{amssymb,latexsym}
\usepackage{amsthm}
\usepackage{color}
\usepackage{graphicx}
\usepackage{amscd}
\usepackage{hyperref}

\title{Linear sets and MRD-codes arising from a class of scattered linearized polynomials}
\author{Giovanni Longobardi - Corrado Zanella}

\newcommand{\cC}{{\mathcal C}}

\newcommand{\F}{{\mathbb F}}

\newcommand{\Fq}{{\F_q}}

\newcommand{\Fqn}{\F_{q^n}}

\newcommand{\Fqt}{{\F_{q^t}}}

\newcommand{\la}{\langle}
\newcommand{\ra}{\rangle}
\newcommand{\im}{\textnormal{im}}
\renewcommand{\mod}{\hbox{{\rm mod}\,}}

\newtheorem{theorem}{Theorem}[section]
\newtheorem{lemma}[theorem]{Lemma}
\newtheorem{corollary}[theorem]{Corollary}
\newtheorem{proposition}[theorem]{Proposition}

\DeclareMathOperator{\tr}{Tr}
\newcommand{\Trnt}{\tr_{q^n/q^t}}
\DeclareMathOperator{\PG}{{PG}}

\DeclareMathOperator{\GL}{{GL}}

\DeclareMathOperator{\PGaL}{P\Gamma L}
\DeclareMathOperator{\GaL}{\Gamma L}

\DeclareMathOperator{\rk}{rk}
\DeclareMathOperator{\N}{N}

\DeclareMathOperator{\End}{End}
\DeclareMathOperator{\intn}{intn}

\theoremstyle{definition}

\newtheorem{remark}[theorem]{Remark}

\begin{document}

\maketitle
\begin{abstract}
	A class of scattered linearized polynomials covering infinitely many field extensions is exhibited.
	More precisely, the $q$-polynomial over $\F_{q^6}$, $q \equiv 1\pmod 4$ described in \cite{ZZ,BZZ}
	is generalized for any even $n\ge6$ to an $\Fq$-linear automorphism
	$\psi(x)$ of $\Fqn$ of order $n$.
	Such $\psi(x)$ and some functional powers of it are proved to be scattered.
	In particular this  provides new maximum scattered linear sets of the projective line 
	$\PG(1,q^n)$ for $n=8,10$.
	The polynomials described in this paper lead to  a new infinite family 
	of MRD-codes in $\F_q^{n\times n}$ with minimum distance $n-1$ for any odd $q$ if $n\equiv0\pmod4$
	and any $q\equiv1\pmod4$ if $n\equiv2\pmod4$.
\end{abstract}

\section{Introduction and preliminaries}

Let $\Fqn$ be the Galois field of order $q^n$, $q$ a prime power. An $\Fq$-\emph{linearized polynomial}, or \emph{$q$-polynomial}, over $\Fqn$ is a polynomial of the form
$$f(x) =\sum_{i=0}^{k}c_i x^{q^i}\in\Fqn[x], \quad k \in \mathbb{N}.$$ 
If $c_k \neq 0$, the integer $k$ is called the \emph{$q$-degree} of $f$, 
in short $\deg_q(f)$.
It is well known that any linearized polynomial defines an endomorphism of $\F_{q^n}$, when $\F_{q^n}$ is regarded as an $\F_q$-vector space and, vice versa, each element of $\End_{\F_q}(\F_{q^n})$	can be represented as a unique linearized polynomial over $\F_{q^n}$ of $q$-degree less than $n$, see \cite{lidl}.\\
For a $q$-polynomial $f(x) = \sum_{i=0}^{n-1}c_i x^{q^i}$ 
over $\Fqn$, let $D_f$ denote the associated \textit{Dickson matrix} (or $q$-\textit{circulant matrix} )
\begin{equation}\label{dickson}
D_f=
\begin{pmatrix}

a_0 & a_1 &\ldots & a_{n-1}\\
a^q_{n-1} & a^q_0 & \ldots & a^q_{n-2}\\
\vdots & \vdots & \vdots & \vdots \\
a_1^{q^{n-1}} & a_{2}^{q^{n-1}} & \ldots & a_0^{q^{n-1}}
\end{pmatrix}.
\end{equation}

The rank of the matrix $D_f$ is the rank of the  $\Fq$-linear map $f(x)$, see \cite{WuLiu}.\\
Among the linearized polynomials over a finite field, a particular class has recently aroused interest for its connections with finite geometry and with coding theory: that of scattered polynomials.

More precisely, a  \textit{scattered} $q$-\textit{polynomial} $f(x) \in \Fqn[x]$ has the property that the polynomial
$f(x)+mx$
has at most $q$ roots in $\Fqn$ for all $m \in \Fqn$, or equivalently, if  for any $y,z \in \Fqn^*$ 
the condition
\begin{equation}\label{scattered}
\frac{f(y)}{y}=\frac{f(z)}{z}
\end{equation}
implies that $y$ and $z$ are $\Fq$-linearly dependent.
The condition for a $q$-polynomial to be scattered can be equivalently stated in terms of Dickson matrices 
\cite{Cs2018,Z}.
Polynomials of this sort are linked to particular subsets of the finite projective line $\PG(1, q^n)$ called \textit{maximum scattered linear sets}.
Then consider  the finite projective line
$\Lambda=\PG(\Fqn^2, \Fqn ) \cong \PG(1, q^n)$. 
A set $L$ of points 
in $\Lambda$  is called $\Fq$-\textit{linear set}
(or just \emph{linear set}) of \textit{rank} $k$ if it consists of the points 
defined by the nonzero vectors of an $\Fq$-subspace $U$ of $\Fqn^2$ of dimension $k$, i.e.
$$L=L_U=\{\la \textbf{u} \ra_{\Fqn} \,:\, \textbf{u} \in U \setminus \{\textbf{0}\} \}.$$
Two linear sets $L_U$ and $L_W$ of $\PG(1, q^n)$ are said to be $\PGaL$-\textit{equivalent} if there is an 
element $\varphi \in \PGaL(2, q^n)$ such that $L_U^\varphi = L_W$. It is clear that if $U$ and $W$ are on the 
same $\GaL(2,q^n)$-orbit, then $L_U$ and $L_W$ are $\PGaL$-equivalent, but the converse statement is not true 
in general.
For further details see \cite{CMP,CsZ2015}.

The set $L_U$ is called $\Fq$-linear set of $\mathcal{Z}(\GaL)$-\textit{class} $r$ if $r$ is the greatest 
integer such that there exist $\Fq$-subspaces $U_1,U_2,\ldots,U_r$ of $\Fqn^2$ such that $L_{U_i}= L_U$ for 
$i \in\{ 1,2 ,\ldots,r\}$ and $U_i \not = \lambda U_j$ for any $\lambda \in \Fqn^*$ and distinct 
$i,j \in \{1,2,\ldots,r\}$. 
Furthermore, $L_U$ is of $\GaL$-\textit{class} $s$ if $s$ is the greatest integer such that there exist 
$\Fq$-subspaces $U_1,U_2,\ldots,U_s$ of $\Fqn^2$ with $L_{U_i}=L_U$ for $i \in \{1,2,\ldots,s\}$, but $U_i$ 
and $U_j$ are not on the same $\GaL(2,q^n)$-orbit for $i,j\in\{1,2,\ldots,s\}$, $i\neq j$.
In particular, if $s=1$, then 	$L_U$ is called a \textit{simple} $\Fq$-linear set.

The scattered $q$-polynomials arise from some $\Fq$-linear sets  in $\PG(1,q^n)$. 
An $\Fq$-linear set of rank $k$ and size $(q^k-1)/(q-1)$ in $\PG(1,q^n)$ is called
\textit{scattered}.
A scattered $\Fq$-linear set of rank $n$ is  called \textit{maximum scattered linear set}.
As shown in \cite{BL2000}, these  are the linear sets of maximum size distinct from $\PG(1, q^n)$.
If $L_U$ is an $\Fq$-linear set of rank $n$ of $\PG(1, q^n)$,  it can always be assumed (up to a projectivity) 
that $L_U$ does not contain the point $\la (0, 1) \ra_{\Fqn}$. 
Then $U = U_f = \{(x, f(x)): x \in \Fqn \}$, 
for some  $q$-polynomial $f(x)$ over $\Fqn$ and for the sake of simplicity, we will write $L_f$ instead of $L_{U_f}$ to denote the linear set defined by $U_f$.
Clearly, $L_f$ is scattered if and only if $f(x)$ is a scattered $q$-polynomial.
The first examples of scattered polynomials were found by Blokhuis and Lavrauw in \cite{BL2000}  and
by Lunardon and Polverino in \cite{LP2001} and then generalized by Sheekey in \cite{Sh}. 
Apart from these, very few examples are known. 
They are defined  for $n\le8$ and are summarized in Section  \ref{newLS}.
In view of the results in \cite{BarMon,BartoliZhou}, stating that the only families of scattered
$q$-polynomials defined for infinitely many $n$ and satisfying certain additional assumptions 
are those of Blokhuis-Lavrauw and Lunardon-Polverino,
it would seem that the scattered polynomials are quite rare.

As stated before, scattered polynomials attracted a lot of attention, especially because of their connection, established by Sheekey in \cite[Section 5]{Sh}, with rank distance codes.
These were introduced by Delsarte as $q$-analogs of the usual linear error correcting codes endowed with Hamming distance, \cite{Delsarte}.
Recently there has been a resurgence of interest in them 
because of their applications to random linear network coding and cryptography, see \cite{res1,res2}.
A \textit{rank distance code} (or RD-code for short) $\cC$ is a subset of the set of $m \times n$ matrices 
$\F^{m \times n}_q$ over
$\Fq$, the finite field of $q$ elements with $q$ a prime power, endowed with the distance function
$$d(A, B) = \rk (A - B)$$
for any $A, B  \in \F_q^{m \times n}$. The \textit{minimum distance} of an RD-code $\cC$, $|\cC|\geq 2$, is defined as
\begin{equation*}
d(\cC) = \min_{\underset{M \not = N}{M, N \in \cC}} d(M, N) \,.
\end{equation*}
A rank distance code of $\F_q^{m \times n }$ with minimum distance $d$ has \emph{parameters}
$(m, n, q; d)$. 
If  $\cC$ is an $\F_q$-linear subspace of $\F_q^{m \times n}$, then $\cC$ is called $\F_q$-\textit{linear} 
RD-code and its \emph{dimension} $\dim_{\F_q} \cC$ is defined to be the dimension of $\cC$ as 
a subspace over $\F_q$.

The \textit{Singleton-like bound}  \cite{Delsarte} for an $(m,n,q;d)$ RD-code $\cC$ is
$$|\cC| \leq q^{\max\{m,n\}(\min\{m,n\}-d+1)}.$$
If the size of the code $\cC$ meets this bound, then $\cC$ is a called  \textit{Maximum
	Rank Distance code}, MRD-\textit{code} for short.
In this paper only the case $m=n$ is considered; that is, only codes whose codewords are square matrices
are taken into account. 
Note that if $n = d$, then an MRD-code $\cC$ consists of $q^n$ invertible endomorphisms of $\Fqn$; such $\cC$  is called \textit{spread set} of $\End_{\Fq}
(\Fqn )$. 
In particular if $\cC$ 
is $\Fq$-linear, it is called a \textit{semifield spread set} of $\End_{\Fq}(\Fqn )$. 
These objects are related to semifields. For more details see \cite{LavPol,LZ}.

The \emph{adjoint code} of a rank code $\cC$ is $\cC^\top=\{C^t\colon C\in\cC\}$,
where $C^t$ denotes the transpose of the matrix $C$.
Two $\Fq$-linear codes $\cC$ and $\cC'$ are called \textit{equivalent} if
there exist $P,Q \in \GL(n, q)$ and a field
automorphism $\sigma$ of $\Fq$ such that 
$$\cC' = \{PC^\sigma Q \colon C \in \cC\}.$$
Furthermore, $\cC$ and $\cC'$ are \emph{weakly equivalent} if $\cC$ is equivalent to
$\cC'$ or to $(\cC')^\top$.

In general, it is difficult to decide whether two rank distance codes with the same
parameters are equivalent or not. 
Useful tools to face this problem are the left and right idealisers, 
see \cite{liebhold_automorphism_2016, LunTromZhou}. 
More precisely, the
\textit{left} and \textit{right idealisers} $L(\cC)$ and $R(\cC)$
of an RD-code $\cC \subseteq  \F_q^{n \times n}$ are
\[L(\cC) = \{X \in \F_q^{n\times n} \colon X C  \in \cC \,\,\textnormal{for all} \,\, C \in \cC\},\]
\[R(\cC) = \{Y \in \F_q^{n\times n} \colon  C Y  \in \cC \,\,\textnormal{for all}\,\, C \in \cC\},\]
respectively.

In this article a class of scattered linearized polynomials over $\Fqn$ will be introduced. Later, the connections to maximum scattered linear sets of the projective line and MRD-codes that arise from these polynomials  will be investigated. More precisely, in Theorem \ref{t:main}, it will proved that the polynomial 
\begin{equation}\label{e:primopsi}
2\psi(x)=x^q+x^{q^{t-1}}-x^{q^{t+1}}+x^{q^{2t-1}}\in\F_{q^{2t}},\quad t\ge3,
\end{equation}
is scattered for any odd $q$ if $t$ is even, and for $q\equiv1\pmod4$ if $t$ is odd. 
Some compositions of type $\psi\circ\psi\circ\cdots\circ\psi(x)$ are scattered as well.
For $t=3$, $\psi(x)$ is up to equivalence the polynomial dealt with in \cite{ZZ,BZZ}.

This paper is organized as follows.
In Section \ref{scattered-class}, the polynomials of type \eqref{e:primopsi} are investigated.
In Section \ref{newLS}, it is shown that this family of scattered polynomials provides maximum scattered 
linear 
sets that are not of pseudoregulus type for any even $n \geq 6$.
The related linear sets in $\PG(1,q^n)$ are proved to be $\PGaL$-equivalent to no previously known
linear set for $n=8,10$.

In the last Section, Sheekey's connection with the MRD-codes is described.
In Theorem \ref{t:finale} it is proved 
that the class
\eqref{e:primopsi} of linearized polynomials gives rise to maximum subsets of square matrices of any even 
order $n \geq 6$ with minimum rank distance $d=n-1$.
They are  not equivalent to
any previously known MRD-code.

\section{\texorpdfstring{A class of scattered $q$-polynomials}{A class of scattered q-polynomials}}
\label{scattered-class}

Throughout this paper $q$ denotes a power of a prime $p\neq2$,  $t \geq 3$, $t\in\mathbb N$, and $n=2t$. 
As usual, if $\ell$ divides $m$,
\[
\tr_{q^m/q^\ell}(x)=x+x^{q^\ell}+x^{q^{2\ell}}+\cdots+x^{q^{m-\ell}}\quad\mbox{and}\quad
\N_{q^m/q^\ell}(x)=x^{\frac{q^{m}-1}{q^{\ell}-1}}
\]
denote the \emph{trace} and the \emph{norm} of $x\in\F_{q^{m}}$ over $\F_{q^{\ell}}$.
Consider the following $q$-polynomials in $\Fqn[x]$: 
\[
\alpha_n(x)=\frac{\Trnt(x)^{q^{t-1}}}2 \quad \textnormal{and}\quad \beta_n(x)=\frac{(x-x^{q^t})^q}2.
\]
In the following the index $n$ will usually be omitted: $\alpha(x)=\alpha_n(x)$, $\beta_n(x)=\beta(x)$.
Note that $\alpha(x)^{q^t}=\alpha(x)$, $\beta(x)^{q^t}=-\beta(x)$ for any $x\in\Fqn$.

\begin{proposition}\label{simple_lemma}
	Let $\Fqn$ be the finite field of order $q^n$, $n=2t$, and consider
	$$W=\{x\in\Fqn\colon x+x^{q^t}=0\}.$$
	Then
	\begin{enumerate}[(i)]
		\item $\ker \alpha = \im \beta=W$;
		\item $\ker \beta=\im \alpha=\Fqt$;
		\item the additive group of $\Fqn$ is direct sum of $\Fqt$ and $W$;
		\item the product of two elements in $W$ is in $\Fqt$;
		\item For any $k\in\mathbb N$, $q\equiv 1 \pmod 4$, and $x\in\Fqn$, $x^{q^k+1}=1$ 
		 implies  $x\not\in W$.
	\end{enumerate}
\end{proposition}
\begin{proof}
Only the last statement is non-trivial.
Let $\omega$ be a generator of the multiplicative group $\F_{q^n}^*$.
Then
\[
W=\{\omega^{(2\ell+1)(q^t+1)/2}\colon \ell=1,2,\ldots,q^t-1\}\cup\{0\}.
\]
In order to be a $(q^k+1)$-th root of the unity, the generic element of $W$ above must satisfy
\[
\frac{(2\ell+1)(q^t+1)}2(q^k+1)\equiv0\pmod{q^{2t}-1},
\]
whence an integer $m$ exists satisfying
\begin{equation}\label{e:con}
2m(q^t-1)=(2\ell+1)(q^k+1).
\end{equation}
Therefore, 4 must divide $q^k+1$. This implies that $q \equiv -1$ $\pmod 4$, a contradiction.\\
\end{proof}


\begin{remark}\label{r:qk1}
If $t$ and $k$ are odd and $q\equiv 3 \pmod 4$, then $(q^t-1)/2$ is odd and $(q^k+1)/4$ is an integer.
This implies that \eqref{e:con} has a solution with $2\ell+1=(q^t-1)/2$ and $m=(q^k+1)/4$.
Therefore, an $x\in W$ exists such that $x^{q^k+1}=1$.
\end{remark}

Now consider the $q$-polynomial
\begin{equation}\label{candidato}
\psi_n(x)=\alpha_n(x)+\beta_n(x).
\end{equation}
The index $n$ will be often omitted in what follows.

For any positive integer, 
\[
\psi^{(k)}(x)=\overbrace{\psi\circ\psi\circ\cdots\circ\psi}^k(x)
\]
will denote the $k$-fold
composition of $\psi$ with itself.
The polynomials $\alpha^{(k)}(x)$ and $\beta^{(k)}(x)$ are defined analogously.
All $q$-polynomials will be considered here as maps; that is, they are reduced modulo $x^{q^n}-x$.

\begin{proposition}
	The map $\psi(x)$ has order $n$.
\end{proposition}

\proof First, $\alpha^{(k)}(x)=\alpha(x)^{q^{(k-1)(t-1)}}$ and $\beta^{(k)}(x)=\beta(x)^{q^{k-1}}$
 for any $k\in\mathbb N$, whence
\begin{equation}\label{alphaadn}
\alpha^{(n)}(x)=\Trnt(x)/2\quad\mbox{ and }\quad\beta^{(n)}(x)=\frac{x-x^{q^t}}2.
\end{equation}
Note that for any $k \in \mathbb{N}$ and $x\in\Fqn$, $\alpha^{(k)}(x) \in \Fqt$ and $\beta^{(k)}(x) \in W$.\\
Next, we prove by induction that for any $k \in \mathbb{N}$
\begin{equation}\label{k-iteration}
\psi^{(k)}(x)=\alpha^{(k)}(x)+ \beta^{(k)}(x).
\end{equation}
The induction base step ($k=1$) is clear by the definition of $\psi(x)$. 
Now suppose that the property in \eqref{k-iteration} holds for $k-1$; then
\begin{equation}
\begin{split}
\psi^{(k)}(x)= &\psi(\psi^{(k-1)}(x))=\psi(\alpha^{(k-1)}(x))+\psi(\beta^{(k-1)}(x))=\\
& \alpha(\alpha^{(k-1)}(x))+\beta(\alpha^{(k-1)}(x))+\alpha(\beta^{(k-1)}(x))+\beta(\beta^{(k-1)}(x)).
\end{split}
\end{equation}
By Proposition $\ref{simple_lemma}$ $(i)$ $(ii)$, we get $\eqref{k-iteration}$, that in view of \eqref{alphaadn} implies
$\psi^{(n)}(x)=x$.
\endproof
As a consequence of  \eqref{k-iteration},
    \begin{equation}\label{e:kfold}
    \psi^{(k)}(x)=\frac12 \bigl (x^{q^k}+x^{q^{t-k}}-x^{q^{t+k}}+x^{q^{2t-k}} \bigr )
    \end{equation}
    for any $0\leq k\leq t$; \eqref{e:kfold} can be extended to any $k\in\mathbb N$ by considering modulo $2t$
the exponents of $q$.\\

\begin{theorem}\label{t:main}
    Let $q$ be an odd prime power, $t\ge3$, and
    \[
    \psi(x)=\frac12 \bigl (x^{q}+x^{q^{t-1}}-x^{q^{t+1}}+x^{q^{2t-1}} \bigr )\in\F_{q^{2t}}[x].
    \]
    For $1\le k<2t$, the $k$-fold composition 
    $\psi^{(k)}(x)$ 
    is scattered if and only if one of the following holds: $(i)$ $t$ is even and $\gcd(k,t)=1$, or
	$(ii)$ $t$ is odd,  $\gcd(k,2t)=1$, and $q\equiv 1\pmod4$.
\end{theorem}
\begin{proof}
The first part of the proof is devoted to prove that any of the conditions $(i)$ and $(ii)$ implies that 
$\psi^{(k)}(x)$ is scattered.
It is straightforward to see that the condition for $\psi^{(k)}(x)$ to be scattered can be 
rephrased in this way: if
\begin{equation}\label{e:psirx}
\psi^{(k)}(\rho x)=\rho\psi^{(k)}(x),\quad x,\rho\in\Fqn,\ x\neq0,
\end{equation}
then $\rho\in\Fq$.
By Lemma \ref{simple_lemma}~$(iii)$, 
for any $\rho\in\Fqn$ there are precisely two elements $h=h_\rho\in\Fqt$ and $r=r_\rho\in W$ such that 
$\rho=h+r$.
Condition \eqref{e:psirx} is equivalent to
\begin{equation}\label{notsplit}
\alpha^{(k)}((r+h)(x_1+x_2))+\beta^{(k)}((r+h) (x_1+x_2))= (r+h)(\alpha^{(k)} ( x_1+x_2)+\beta^{(k)}
(x_1+x_2)),
\end{equation}
with $x=x_1+x_2$, where $x_1\in \F_{q^t}$ and $x_2$ belongs to $W$.\\
By Proposition \ref{simple_lemma}~$(i)\ (ii)\ (iv)$, the expression in \eqref{notsplit} is reduced to
\begin{equation}\label{e:reduced}
\alpha^{(k)}(rx_2)+\alpha^{(k)}(hx_1)+\beta^{(k)}(rx_1)+\beta^{(k)}(hx_2)= (r+h)(\alpha^{(k)} ( x_1)+\beta^{(k)}(x_2)).
\end{equation}
Now, by expanding \eqref{e:reduced},
\begin{equation}
(rx_2)^{q^{k(t-1)}}+(hx_1)^{q^{k(t-1)}}+(rx_1)^{q^k}+(hx_2)^{q^k}= (r+h)( x^{q^{k(t-1)}}_1+x^{q^k}_2).
\end{equation}
Since $W$ and $\F_{q^t}$ meet in the trivial space, one obtains
\begin{equation*}
\begin{cases}
rx^{q^k}_2-r^{q^{k(t-1)}}x_2^{q^{k(t-1)}}=(h^{q^{k(t-1)}}-h)x^{q^{k(t-1)}}_1\\
r^{q^k}x^{q^k}_1-rx^{q^{k(t-1)}}_1=(h-h^{q^k})x^{q^k}_2.
\end{cases}
\end{equation*}
Raising the first equation to its $q^k$-power,
\begin{equation}\label{system}
\begin{cases}
r^{q^k}x^{q^{2k}}_2-rx_2=(h-h^{q^k})x_1\\
r^{q^k}x^{q^k}_1-rx^{q^{k(t-1)}}_1=(h-h^{q^k})x^{q^k}_2.
\end{cases}
\end{equation}
This can be seen as a linear system in the unknowns $r$ and $r^{q^k}$.
In the following it will assumed that $r\neq0$, leading to a contradiction.

\begin{itemize}
	\item [-] \textsl{Case 1.} $x_1=0$: the linear system in \eqref{system} is reduced to
	\begin{equation}\label{system:x1=0}
	\begin{cases}
	r^{q^k}x^{q^{2k}}_2-rx_2=0\\
	(h-h^{q^k})x^{q^k}_2=0
	\end{cases}
	\end{equation}
	Since $x_2 \neq 0$, the first equation in \eqref{system:x1=0}
	gives  $r^{q^k-1}=(x_2^{-1})^{q^{2k}-1}$.
	Then there exists $\mu\in\F_{q^k}^*\cap\Fqn=\F_q^*$ such that
	\[
	r=\mu(x_2^{-1})^{q^k+1}.
	\]
	The right hand side of the last equation belongs to $\F_{q^t}^*$, that is a contradiction.
	\item[-] \textsl{Case 2.} $x_2=0:$ as before, the linear system in \eqref{system} is reduced to
	\begin{equation}\label{system:x2=0}
	\begin{cases}
	(h-h^{q^k})x_1=0\\
	r^{q^k}x^{q^k}_1-rx^{q^{k(t-1)}}_1=0
	\end{cases}
	\end{equation}
	Since $x_1 \neq 0$, by the second equation in \eqref{system:x2=0}, the existence of a 
	$\lambda\in\F_q^*$ follows such that
	$r=\lambda(x_1^{q^k})^{q^{k(t-3)}+q^{k(t-4)}+\ldots+q^k+1}$, implying $r\in\F_{q^t}^*$, a contradiction.
	
	Note that (again under the assumption $r\neq0$) $h\not\in\Fq$, since for $h\in\Fq$ the same arguments as above lead to a contradiction.

	\item [-] \textsl{Case 3.} $x_1\neq0\neq x_2$.
	The first goal is to prove that the determinant 
	$D=x^{q^k}_1x_2-x^{q^{k(t-1)}}_1x^{q^{2k}}_2$ is not zero.
		
	If $t$ is even, $D$ cannot be zero, otherwise the existence of a 
	$\lambda \in \F_{q^2}$ would follow 
	such that $x_2=\lambda(x_1^{-q^k})^{q^{k(t-4)}+\ldots+ q^{2k}+1}$, implying 
	$x_2 \in \Fqt$, a contradiction.
	
	If $t$ is odd, $q\equiv 1\pmod4$, and $D=0$, a necessary condition for \eqref{system} to
	have a solution is
	\[
	\det\begin{pmatrix} x_2^{q^{2k}}&(h-h^{q^k})x_1\\ x_1^{q^k}&(h-h^{q^k})x_2^{q^k}\end{pmatrix}=0,
	\]
	leading to $(x_2^{q^k}/x_1)^{q^k+1}=1$.
	Since $x_2^{q^k}/x_1\in W$, this contradicts Proposition  \ref{simple_lemma}~(\emph v).

	Therefore, $D\neq0$. Obtaining $r^{q^k}$ and $r$ from $\eqref{system}$, 
	
\[\begin{split}
r^{q^k}=(h-h^{q^k})\frac{x^{q^k+1}_2-x^{q^{k(t-1)}+1}_1}{x^{q^k}_1x_2-x^{q^{k(t-1)}}_1x^{q^{2k}}_2}
\quad\mbox{and}\\
r=(h-h^{q^k})\frac{x_2^{q^{2k}+q^k}-x^{q^k+1}_1}{x_1^{q^k}x_2-x^{q^{k(t-1)}}_1x^{q^{2k}}_2}.
\end{split}\]
	Therefore,
	$$r^{q^k-1}=\frac{x^{q^k+1}_2-x^{q^{k(t-1)}+1}_1}{x_2^{q^{2k}+q^k}-x^{q^k+1}_1}.$$
	Note that $y=x_2^{q^{2k}+q^k}-x^{q^k+1}_1$ belongs to $\Fqt$, and
	$r^{q^k-1}=y^{q^{k(t-1)}-1}$.
	Then there exists $\lambda \in \F_{q}$ such that 
	$r=\lambda y^{q^{k(t-2)}+q^{k(t-3)}+\cdots+q^k+1}$, implying $r \in \Fqt$, a contradiction again.
	
	Hence, the condition $r\neq0$ yields in all cases a contradiction.
	Taking into account \eqref{system}, $h$ must belong to $\Fq$, implying $\rho\in\Fq$.
	This concludes the proof of the sufficiency.	
\end{itemize}

If $d=\gcd(k,t)\neq1$ then $\psi^{(k)}(x)$ is a $q^d$-polynomial, hence it is not scattered
as a $q$-polynomial. So, $\gcd(k,t)=1$ is a necessary condition.

Assume $k$ is even; then it may be assumed that $t$ is odd.
An $x\in\F_{q^2}^*\subseteq\F_{q^n}$ belongs to $W$ if and only if $x^{q-1}=-1$ which is an equation
admitting $(q-1)$ solutions. Then fix $\mu\in W \cap \F_{q^2}^*\subseteq\F_{q^n}$ and 
$x_2\in W\setminus\{0\}$ arbitrarily.
A solution of \eqref{system} is $x_1=h=0$, $r=\mu(x_2^{-1})^{q^k+1}$ and this implies that
$\psi^{(k)}(x)$ is not scattered.

For odd $t$ and $k$ and $q\equiv 3\pmod4$, the $q$-polynomial $\psi^{(k)}(x)$ is not scattered. 
Indeed, taking $x_1=1$, $x_2\in W$ such that $x_2^{q^k+1}=1$,
	the equations in \eqref{system} are equivalent for any $h\in\Fqt$.
	The images of the $\Fq$-linear maps $r\in W\mapsto r^{q^k}x_2^{q^{2k}}-rx_2\in \F_{q^t}$
	and $h\in\F_{q^t}\mapsto (h-h^{q^k}) x_1\in \F_{q^t}$ both have $\Fq$-dimension at least $t-1$;
	this implies that their intersection is not trivial, and $r\in W$, $h\in\F_{q^t}$
	exist such that $r\neq0$ and \eqref{system} is satisfied.
\end{proof}

		\begin{remark}
			Note that for $n=6$, $2\psi^{(5)}(x)$ is the polynomial described in \cite{ZZ}, that for
			$q\equiv 1\pmod4$ and it is associated with the scattered $\Fq$-linear set $L_h^{5,6}\subseteq\PG(1,q^6)$,
			$h^2=-1$ that will be described in Section \ref{newLS}. Furthermore,
			$\psi^{(5)}(x)$ and $\psi(x)$ determine the same linear set, since $\psi^{(5)}(x)$ is the adjoint map of 
			$\psi(x)$ with respect to the bilinear form $\tr_{q^6/q}(xy)$  (cf.\ \cite{BaGiMaPo,CMP}).
\end{remark}

\section{On the equivalence issue for linear sets }\label{newLS} 

By definition, 
\[
L_{\psi_n^{(k)}}=\{\la(x,\psi_n^{(k)}(x))\ra_{\Fqn}\colon x\in\F_{q^n}^*\}
\]
denotes the maximum scattered $\Fq$-linear set of $\PG(1,q^n)$ associated with
$\psi_n^{(k)}(x)\in\F_{q^n}[x]$, provided that the assumptions of Theorem \ref{t:main} are satisfied.
The related $\Fq$-vector subspace of $\F_{q^n}^2$ is
\[
U_{\psi_n^{(k)}}=\{(x,\psi_n^{(k)}(x))\colon x\in\F_{q^n}\}.
\]
Since the collineation $(a,b)\in\F_{q^n}^2\mapsto(a^p,b^p)\in\F_{q^n}^2$ stabilizes $U_{\psi_n^{(k)}}$,
an $\Fq$-subspace of  $\F_{q^n}^2$ is $\GaL(2,q^n)$-equivalent to $U_{\psi_n^{(k)}}$ if and only if
it is $\GL(2,q^n)$-equivalent to $U_{\psi_n^{(k)}}$.

\begin{proposition}\label{GAL_psi}
Let $t\ge3$ and assume that $\psi_n^{(k)}(x)$ and $\psi_n^{(m)}(x)$ are scattered, $1\le k,m<2t=n$.
Then the $\Fq$-subspaces $U_{\psi_n^{(k)}}$ and $U_{\psi_n^{(m)}}$ are equivalent under the action
of $\GaL(2,q^n)$ if and only if $k=m$ or $k=n-m$.
\end{proposition}
\begin{proof}
Since
\[
U_{\psi_n^{(k)}}=\{(\psi_n^{(n-k)}(x),x)\colon x\in\F_{q^{n}}\},
\]
$U_{\psi_n^{(k)}}$ and $U_{\psi_n^{(n-k)}}$ are equivalent under the action
of $\GaL(2,q^n)$.
This allows to prove the necessity of the condition only for $1\leq k <m<t$.

Assume that $\begin{pmatrix}a&b\\ c&d\end{pmatrix}$ is an invertible matrix in $\F_{q^{n}}^{2\times2}$
such that for any $x\in\F_{q^{n}}$ a $z\in\F_{q^{n}}$ exists such that
\begin{equation}\label{e:nonsono}
\begin{pmatrix}a&b\\ c&d\end{pmatrix}\begin{pmatrix}x\\ \alpha^{(k)}(x)+\beta^{(k)}(x)\end{pmatrix}=
\begin{pmatrix}z\\ \alpha^{(m)}(z)+\beta^{(m)}(z)\end{pmatrix}.
\end{equation}
This implies
\begin{gather*}
cx+d\alpha^{(k)}(x)+d\beta^{(k)}(x)=\\
\alpha^{(m)}(ax+b\alpha^{(k)}(x)+b\beta^{(k)}(x))+\beta^{(m)}(ax+b\alpha^{(k)}(x)+b\beta^{(k)}(x))
\end{gather*}
for any $x\in\F_{q^{n}}$.
Decompose any $x\in\F_{q^{n}}$ as a sum $x=x_1+x_2$ with $x_1\in\Fqt$, $x_2\in W$.
The above equation splits in
\begin{equation}\label{e:nonsono2}
\begin{cases}
c_1x_1+c_2x_2+d_1x_1^{q^{t-k}}+d_2x_2^{q^k}=
(a_1x_1+a_2x_2+b_1x_1^{q^{t-k}}+b_2x_2^{q^k})^{q^{t-m}}\\
c_2x_1+c_1x_2+d_2x_1^{q^{t-k}}+d_1x_2^{q^k}=
(a_2x_1+a_1x_2+b_2x_1^{q^{t-k}}+b_1x_2^{q^k})^{q^{m}}.
\end{cases}
\end{equation}
Putting $x_2=0$ in \eqref{e:nonsono2} one obtains that for any $x_1\in\Fqt$
\begin{equation}\label{e:idp1}
\begin{cases}
c_1x_1+d_1x_1^{q^{t-k}}-a_1^{q^{t-m}}x_1^{q^{t-m}}-b_1^{q^{t-m}}x_1^{q^{n-k-m}}=0\\
c_2x_1+d_2x_1^{q^{t-k}}-a_2^{q^{m}}x_1^{q^{m}}-b_2^{q^m}x_1^{q^{t+m-k}}=0.
\end{cases}
\end{equation}
Similarly
\begin{equation}\label{e:idp2}
\begin{cases}
c_2x_2+d_2x_2^{q^k}-a_2^{q^{t-m}}x_2^{q^{t-m}}-b_2^{q^{t-m}}x_2^{q^{t+k-m}}=0\\
c_1x_2+d_1x_2^{q^k}-a_1^{q^{m}}x_2^{q^{m}}-b_1^{q^{m}}x_2^{q^{k+m}}=0
\end{cases}
\end{equation}
must hold for any $x_2\in W$.
After reducing modulo $x^{q^t}-x$ in \eqref{e:idp1} and modulo $x^{q^t}+x$ in
\eqref{e:idp2}, four polynomials are obtained all whose coefficients must be zero.

If $k+m\neq t$, the first identity in \eqref{e:idp1} implies $a_1=b_1=c_1=d_1=0$, and the
first of \eqref{e:idp2} implies $a_2=b_2=c_2=d_2=0$, a contradiction.

Finally assume $k+m=t$.
The second of \eqref{e:idp1} gives $b_2=c_2=0$ and $d_2=a_2^{q^{m}}$;
the first of \eqref{e:idp2} implies $d_2=a_2^{q^{t-m}}$.
As a result, $a_2^{q^{2m}-1}=-1$; $a_2\in\F_{q^{4m}}^*\cap\F_{q^{2t}}$.
Since $m$ and $2t$ are relatively prime, $a_2\in\F_{q^4}\setminus\F_{q^2}$, whence
$t\equiv 2\pmod 4$.
The first of \eqref{e:idp1} gives $a_1=d_1=0$ and $c_1=b_1^{q^{t-m}}$;
the second of \eqref{e:idp2} implies $c_1=-b_1^{q^m}$.
As a result, $b_1^{q^{2m}-1}=-1$;
so, $b_1\in\F_{q^{4}}\cap\F_{q^{t}}=\F_{q^2}$ and this is a contradiction.
\end{proof}

A question that remains open is whether $L_{\psi_n^{(k)}}$ and 
$L_{\psi_n^{(m)}}$ are $\PGaL$-equivalent for $1 \leq k < m <t$.

In order to decide whether the linear sets $L_{\psi_n^{(k)}}$ are new, they 
will be compared to the known maximum scattered linear 
sets in $\PG(1,q^n)$.
The known non-equivalent (under the action of  $\GaL(2,q^n)$) \emph{maximum scattered subspaces} 
of $\F^2_{q^n}$,  i.e.\  subspaces defining maximum  scattered linear sets, are listed below.
\begin{itemize} 
	\item [1.] $U^{1,n}_{s}=\{(x,x^{q^s}): x \in \Fqn\}, 1 \leq s \leq n-1,\gcd(s,n)=1$ 
	\cite{BL2000,CsZ20162},
	\item [2.] $U^{2,n}_{s,\delta}=\{(x,\delta x^{q^s}+x^{q^{n-s}}): x \in \Fqn\}$,
	$n\ge4$, $\N_{q^n/q}(\delta) \not \in \{0,1\}$, $\gcd(s,n)=1$ \cite{LP2001,LTZ,Sh},
	\item [3.] $U_{s,\delta}^{3,n}=\{(x,\delta x^{q^s}+x^{q^{s+n/2}}):x \in \Fqn\}, n \in \{6,8\}, \gcd(s,n/2)=1,\N_{q^n/q^{n/2}}(\delta) \not \in \{0,1\}$, for some $\delta$ and $q$ \cite{CMPZ},
	\item [4.] $U_{\delta}^{4,6} = \{(x, x^q + x^{q^3}+ \delta x^{q^5}): x \in \F_{q^6} \}$, $q$ odd and $\delta^2+\delta= 1$, see \cite{CsMZ2018} for $q \equiv 0, \pm 1 (\mod 5)$, and \cite{MMZ} for the remaining congruences of $q$,
	\item [5.] $U_h^{5,6}=\{(x,h^{q-1}x^q-h^{q^2-1}x^{q^2}+x^{q^4}+x^{q^5})\colon x\in\F_{q^6}\}$,
	$h\in\F_{q^6}$, $h^{q^3+1}=-1$, $q$ odd \cite{BZZ,ZZ}.
\end{itemize}

To make notation easier, $L^{i,n}_{s}$, $L^{i,n}_{s,\delta}$, $L^{4,6}_{\delta}$, and $L_h^{5,6}$  will denote the $\Fq$-linear sets defined by $U^{i,n}_{s}$, $U^{i,n}_{s,\delta}$, $U^{4,6}_{\delta}$,
and $U_h^{5,6}$, respectively.
Moreover, the sets $L_s^{1,n}$ and  $L^{2,n}_{s,\delta}$ are called of  \textit{pseudoregulus type} and \emph{LP-type}, respectively.
As noted in the previous section, $L_{\psi_6}=L^{5,6}_h$ 
 where $h^2=-1$ for $q\equiv1$ $\pmod4$.
In order to understand if, under the assumptions of Theorem \ref{t:main}, the maximum scattered linear set  $L_{\psi_n^{(k)}}$ is of pseudoregulus type or of LP-type, some preliminary results have to be retraced.

First of all, Lunardon and Polverino in \cite[Theorem 1 and 2] {LuPo2004}, (see also \cite{LuPoPo2002}) showed that every linear set is  projection of a canonical subgeometry, where a \textit{canonical subgeometry} in $\PG(m - 1, q^n)$  is a linear set $L$ of rank $m$ such that $\langle L\rangle  = \PG(m - 1, q^n)$. In particular, this result in the projective line
case states that for each $\Fq$-linear set $L_U$ of the projective line $\Lambda = \PG(1, q^n)$ of rank $n$
there exists a canonical subgeometry $\Sigma= \PG(n-1, q)$ of $\Sigma^* = \PG(n-1, q^n)$,
and an $(n - 3)$-subspace $\Gamma$ of $\Sigma^*$ with  $\Gamma$ disjoint from $\Sigma$ and $\Lambda$ such that 
$$L_U = \textnormal{p}_{\Gamma,\Lambda}(\Sigma) = \{\langle \Gamma, P \rangle \cap \Lambda: P \in \Sigma\}.$$
In \cite{CsZ20162}, Csajb\'ok and the second author gave a characterization of the linear sets of pseudoregulus type as a particular projection of a canonical subgeometry showing the following 
\begin{theorem}
[\cite{CsZ20162}, Theorem 2.3]\label{t:2.3}
Let $\Sigma$ be a canonical subgeometry of $\PG(n-1, q^n)$, $q > 2$, $n \geq 3$. Assume that $\Gamma$ and $\Lambda$ are an $(n - 3)$-subspace and a
line of $\PG(n -1, q^n)$, respectively, such that $\Gamma \cap \Sigma \not = \emptyset \neq \Gamma \cap \Lambda$. 
Then the
following assertions are equivalent:
\begin{itemize}
	\item [$(i)$] the set $\textnormal{p}_{\Gamma,\Lambda}(\Sigma)$ is a scattered $\Fq$-linear set of pseudoregulus type;
	\item [$(ii)$] there exists a generator $\sigma$ of the subgroup $\PGaL(n,q^n)$ fixing $\Sigma$ pointwise and such that $\dim(\Gamma \cap \Gamma^\sigma)= n -4$;
	\item [$(iii)$] there exist a point $P_\Gamma$ and a generator $\sigma$ of the subgroup of $\PGaL(n, q^n)$ fixing $\Sigma$ pointwise, such that $\langle P_\Gamma, P_{\Gamma}^{\sigma},\ldots,P_\Gamma^{\sigma^{n-1}} \rangle=\PG(n-1,q^n)$, and
	
	$$\Gamma=\langle P_\Gamma, P_\Gamma^{\sigma},\ldots,P_\Gamma^{{\sigma}^{n-3}} \rangle.$$
	
\end{itemize}
\end{theorem}

Therefore, using the Theorem above, one obtains 
 
\begin{proposition}\label{p:s1}
	For any $n\ge6$ and $k$ such that $\gcd(n,k)=1$,
	the linear set $L_{\psi_n^{(k)}}$ is not of pseudoregulus type.
\end{proposition}
\begin{proof}
Since the linear set $L_{2\psi^{(k)}_n}$ can be represented as the projection of the subgeometry $\Sigma$ whose points
	are of type $P_u=\la(u,u^q,\ldots,u^{q^{n-1}})\ra_{\Fqn}$, $u\neq0$, from the vertex $\Gamma_k$
	of equations $x_0=x_k+x_{t-k}-x_{t+k}+x_{n-k}=0$ onto the line 
	$\ell$ of equations $x_1=x_2=\ldots=x_{n-k-1}=x_{n-k+1}=\ldots=x_{n-1}=0$.
	For, the hyperplane of $\PG(n-1,q^n)$ joining $\Gamma_k$ and $P_u$ has equations
	\[
	  2\psi_n^{(k)}(u)x_0-u(x_k+x_{t-k}-x_{t+k}+x_{n-k})=0.
	\]
	Such hyperplane meets $\ell$ in the point all whose coordinates are zero except $x_0=u$,
	$x_{n-k}=2\psi_n^{(k)}(u)$.
	Let 
	\[
	\sigma:\la(x_0,x_1,\ldots,x_{n-1})\ra_{\Fqn}\mapsto\la(x_{n-1}^q,x_0^q,\ldots,x_{n-2}^q)\ra_{\Fqn},
	\]
	then $\sigma$ is a generator of the subgroup of P$\Gamma$L$(n,q)$ fixing $\Sigma$ pointwise.
	Since the dimension of $\Gamma_k\cap\Gamma_k^{\sigma^m}$ is less than $n-4$ for any $m=1,\ldots,n-1$,
	$L_{\psi_n^{(k)}}$ is not of pseudoregulus type by Theorem \ref{t:2.3}. 
\end{proof}

In \cite{ZZ}, the \emph{intersection number} of an $(n-3)$-subspace $\Gamma$ of $\PG(n-1,q^n)$
with respect to a collineation $\sigma$ fixing pointwise a $q$-order subgeometry $\Sigma$ such that 
$\Sigma\cap\Gamma=\emptyset$ has been defined as
\[
\intn_\sigma(\Gamma)=\min\{k\in\mathbb N\colon\dim(\Gamma\cap\Gamma^\sigma\cap\ldots\cap\Gamma^{\sigma^k})
>n-3-2k\}.
\]
By means of this notion and since the linear set $L^{2,n}_{s,\delta}$ has $\GaL$-class at most 2 for $n \in \{5, 6, 8\}$ (see \cite{CsMP2018} and \cite{CsMZ2018}), 
Zullo and the second author gave a characterization in term of projection for linear sets of LP-type. More precisely,
\begin{theorem}[\cite{ZZ}, Theorem 3.5,] \label{Theorem3.5}
Let $L$ be a maximum scattered linear set in $\Lambda = PG(1,q^n)$ with $n \leq 6$ or $n = 8$. Then $L$ is a linear set of LP-type if and only if for each $(n-3)$-subspace $\Gamma$ of $\PG(n-1,q^n)$ such that $L=\textnormal{p}_{\Gamma,\Lambda(\Sigma)}$, the following holds:
\begin{itemize}
	\item [$(i)$] there exists a generator $\sigma$  of the subgroup of $\PGaL(n, q^n)$ fixing $\Sigma$ pointwise, such that $\intn_{\sigma}(\Gamma) = 2$;
	\item [$(ii)$] there exist a unique point $P$ and some point $Q$ of $\PG(n -1, q^n)$ such that
	 $$\Gamma=\langle P, P^{\sigma},\ldots,P^{\sigma^{n-4}},Q \rangle,$$
	 and the line $\langle P^{\sigma^{n-1}},P^{\sigma^{n-3}}\rangle$ meets  $\Gamma$.
	
\end{itemize}

\end{theorem}

In \cite[Section 5]{ZZ}, exploiting the Theorem above, it has been shown that $L_{\psi_6}$ is not equivalent to a linear set of LP-type.
The following proposition involves similar arguments. Clearly, by Proposition \ref{GAL_psi}, hereafter one may suppose $k<t$.

\begin{proposition}	\label{p:s2}
	The linear set $L_{\psi_8^{(k)}}$ is not  of LP-type for $k=1,3$.
\end{proposition}

\proof 
The result will be showed only for $k=1$.  Similar considerations lead to the same result also for $k=3$.
Then, as before, the linear set $L_{\psi_8^{(1)}}$ can be represented as the projection of the subgeometry 
$\Sigma$ whose points
are of type $\la(u,u^q,\ldots,u^{q^{7}})\ra_{\F_{q^8}}$, $u\neq0$, from the vertex $\Gamma$
of equations $x_0=x_1+x_{3}-x_{5}+x_{7}=0$ onto the line $x_1=x_2=\ldots=x_{6}=0$.
Let $\sigma\in \PGaL(8,q^8)$ be defined as 
\[ 
  \la(x_0,x_1,\ldots,x_7)\ra_{\F_{q^8}}^{\sigma}=\la (x^q_1,x^q_2,\ldots,x^{q}_0) \ra_{\F_{q^8}}. 
\]
The collineations $\sigma$, $\sigma^3$, $\sigma^5$ and $\sigma^7$ 
are the only generators of the subgroup of $\PGaL(8,q^8)$ fixing pointwise the subgeometry $\Sigma$.
Consider the subspaces

\begin{equation}
\Gamma^{\sigma}:
\begin{cases}
x_1=0 \\
x_2+x_4-x_6+x_0=0
\end{cases}
\textnormal{and}\quad 
\Gamma^{\sigma^2}:
\begin{cases}
x_2=0\\
x_3+x_5-x_7+x_1=0.
\end{cases}	
\end{equation}
By direct computation, $\dim(\Gamma \cap \Gamma^{\sigma})=3$ and, since $q$ is odd, $\dim(\Gamma \cap \Gamma^{\sigma} \cap \Gamma^{\sigma^2})=1$.

Furthermore, since $\Gamma \cap \Gamma^{\sigma^7}=(\Gamma \cap \Gamma^{\sigma})^{\sigma^7}$ and 
$\Gamma \cap \Gamma^{\sigma^7} \cap \Gamma^{(\sigma^7)^2}=(\Gamma \cap \Gamma^\sigma \cap 
\Gamma^{\sigma^2})^{\sigma^6}$, we have $\dim(\Gamma \cap \Gamma^{\sigma^7})=3$ and $\dim(\Gamma \cap 
\Gamma^{\sigma^7} \cap \Gamma^{(\sigma^7)^2})=1$. 
Hence
\[\mathrm{intn}_{\sigma}(\Gamma)=\mathrm{intn}_{\sigma^7}(\Gamma)\geq 3. \]

A similar argument can be applied also for $\sigma^3$ and $\sigma^5$.
As a consequence, the necessary condition stated in Theorem \ref{Theorem3.5} for a linear set in $\PG(1,q^8)$
to be of LP-type is not satisfied by $L_{\psi_8^{(1)}}$.
\endproof

In \cite{CMPZ}, the scattered subspace $ U^{3,n}_{s,\delta}$ of $\F^2_{q^n}$ is  exhibited for $n \in\{6,8\}$, $s$ coprime to $n$ and under some conditions on $\delta$ and $q$.\\
Moreover, according to \cite[Section 5]{CMPZ}, $U^{3,n}_{s,\delta}$ is $\GL(2, q^n)$-equivalent to $U^{3,n}_{n-s,\delta^{q^{n-s}}}$ and to $U^{3,n}_{s+n/2,\delta^{-1}}$ , thus
it is enough to take into account the linear sets $L^{3,n}_{s,\delta}$ with $s < n/4$, $\gcd(s, n/2) = 1$ 
and hence only with $s = 1$ for $n = 6, 8$.
Finally, the author in \cite[Proposition 4.1 and 4.2]{CsMZ2018}  showed that the $\mathcal{Z}(\GaL)$-class of $L^{3,n}_{1,\delta}$ is
two and $L^{3,n}_{1,\delta}$ is a simple linear set.

\begin{proposition}\label{p:s3}
	The linear set $L_{\psi_8^{(k)}}$, $k=1,3$, is not $\PGaL$-equivalent to $L^{3,8}_{s,\delta}$ for any $s$.
\end{proposition}
\proof By the results in \cite{CMPZ,CsMZ2018} quoted above, 
the linear set $L_{\psi_8^{(1)}}$  is $\PGaL$-equivalent to 
some $L^{3,8}_{s,\delta}$ if and only if 
$$U_{2\psi}=\{(x,x^q+x^{q^3}-x^{q^5}+x^{q^7}): x \in \F_{q^8}\}$$ 
is $\GaL$-equivalent to $U^{3,8}_{1,\delta}$.
Then suppose that there exist an
invertible matrix $\begin{pmatrix}
a & b \\
c & d 
\end{pmatrix}$
such that for each $x \in \F_{q^8}$ there exists $z \in \F_{q^8}$
satisfying
\begin{equation}
\begin{pmatrix}
a & b \\
c & d 
\end{pmatrix}
\begin{pmatrix}
x\\
x^q+x^{q^3}-x^{q^5}+x^{q^7}
\end{pmatrix}=
\begin{pmatrix}
z\\
\delta z^q + z^{q^5}
\end{pmatrix}.
\end{equation}
Equivalently, for each $x \in \F_{q^8}$,
\begin{equation}\label{f}
\begin{gathered}
cx+d(x^q+x^{q^3}-x^{q^5}+x^{q^7})=\\
=\delta [ a^q x^q+b^q(x^{q^2}+x^{q^4}-x^{q^6}+x)]+ 
[ a^{q^5} x^{ q^5}+b^{q^5}(x^{q^6}+x-x^{q^2}+x^{q^4})].
\end{gathered}
\end{equation}

This is a polynomial identity in $x$ that implies
\begin{equation*}
\begin{cases}
c=\delta b^q + b^{q^5}\\
d=\delta a^q\\
0=\delta b^q - b^{q^5}\\
d=0\\
0=\delta b^q + b^{q^5}\\
d=-a^{q^5},
\end{cases}
\end{equation*}
whence, since $\delta \neq 0$, $a=c=d=0$: a contradiction.
By applying the argument above for $U_{2\psi_8^{(3)}}$,
taking into account \eqref{e:kfold},
\begin{gather*}
cx+d(x^q+x^{q^3}+x^{q^5}-x^{q^7})=\\
=\delta [ a^q x^q+b^q(x^{q^2}+x^{q^4}+x^{q^6}-x)]+ 
[ a^{q^5} x^{ q^5}+b^{q^5}(x^{q^6}+x+x^{q^2}-x^{q^4})]. 
\end{gather*}
As before, this polynomial identity in $x$ implies
\begin{equation*}
\begin{cases}
c=-\delta b^q + b^{q^5}\\
d=\delta a^q\\
0=\delta b^q + b^{q^5}\\
d=0\\
0=\delta b^q - b^{q^5}\\
d=a^{q^5},
\end{cases}
\end{equation*}
whence, since $\delta \neq 0$, $a=c=d=0$, a contradiction again.
\endproof

\section{ \texorpdfstring{$\mathcal{Z}(\GaL)$}{Z(GL)}- and \texorpdfstring{$\GaL$}{GL}-class of \texorpdfstring{$L_{\psi_n^{(k)}}$}{L}  for some values of $n$ and $k$}

Now, similarly to what has been done in \cite{CMP}, the $\mathcal{Z}(\GaL)$-class and the $\GaL$-class of the maximum scattered $\Fq$-linear set $L_{\psi_n^{(k)}}$ for small values of $n$ and $k$ will be determined. 
For sake of completeness, the following preliminary results will be recalled.

\begin{proposition}\label{Prop2.3}
	\cite[Proposition 2.3]{CsMZ2018} Let $f$ and $g$ be two $q$-polynomials over $\Fqn$. Then $L_f \subseteq L_g$ if and only if
	$$x^{q^n}- x \,| \, \det D_{F(Y)}(x) \in \Fqn [x],$$
	where $F(Y ) = f(x)Y -g(Y )x$ (cf.\ \eqref{dickson}). 
	In particular, if $\deg(\det D_{F(Y )}(x))< q^n$ , then
	$L_f \subseteq L_g$ if and only if $\det D_{F(Y )}
	(x)$ is the zero polynomial.
\end{proposition}

\begin{lemma}\label{Lemma3.6}
	\cite[Lemma 3.6]{CMP} Let $f(x) =\sum_{i=0}^{k}a_i x^{q^i}$ and $g(x)  =\sum_{i=0}^{k}b_i x^{q^i}$
	be two $q$-polynomials over $\Fqn$ such that $L_f = L_g$. Then
	\begin{equation}\label{(1)}
	a_0 = b_0,
	\end{equation}
	for $k = 1, 2, \ldots,n-1$ it holds that
	\begin{equation}\label{(2)}
	a_ka_{n-k}^{q^k}= b_kb_{{n-k}}^{q^k},
	\end{equation}
	for $k = 2, 3, \ldots , n - 1$ it holds that
	\begin{equation}\label{(3)}
	a_1a^q_{k-1}a_{n-k}^{q^k}+a_ka^q_{n-1}a^{q^k}_{n-k+1} =b_1b^q_{k-1}b_{n-k}^{q^k}+b_kb^q_{n-1}b^{q^k}_{n-k+1} .
	\end{equation}
\end{lemma}

Therefore, the following results can be shown.

\begin{proposition}\label{p:class1}
    Let $q\equiv1\pmod 4$.
	The $\mathcal{Z}(\GaL)$-class of $L_{\psi_{6}}$ is two. 
	Moreover, $L_{\psi_{6}}$ is  a  simple linear set.
\end{proposition}

\begin{proof}
	Since $\psi_{6}(x)$ and $\psi_6^{(5)}(x)$ define the same linear  set, we know that $L_{2\psi_6}=L_{2\psi^{(5)}_{6}}$. Suppose $L_\varphi=L_{2\psi_6}$ for some $\varphi(x)=\sum^5_{i=0}a_{i}x^{q^i} \in \F_{q^6}[x]$. We show that there exists $\lambda \in \F^*_{q^6}$ such that either $\lambda U_\varphi=U_{2\psi_{6}}$ or  $\lambda U_{\varphi}=U_{2\psi^{(5)}_{6}}$.\\
	By \eqref{(1)}, \eqref{(2)} and \eqref{(3)} in Lemma \ref{Lemma3.6}, one obtains that 
	\begin{equation}\label{condition1}
	a_0=a_3=0,\quad a_1a^q_5=1 \quad \text{and} \quad  a_2a_4^{q^2}=-1.
	\end{equation}
	By Proposition \ref{Prop2.3}, the Dickson matrix associated with the $q$-polynomial
	\begin{equation}
	F(Y)=\varphi(x)Y-2\psi_6(Y)x
	\end{equation}
	has determinant $D_{F(Y)}(x)$ equal to zero for each $x \in \F_{q^6}$, where, by \eqref{condition1}, $$\varphi(x)=a_1x^q+a_2x^{q^3}-a^{-q^4}_2x^{q^4}+a^{-q^{5}}_1x^{q^5}.$$
	Direct computation shows that 
	\begin{equation}\label{DF(Y)}
	D_{F(Y)}(x)=\frac{1}{\N_{q^6/q}(a_1a_2)}Q_{a_1,a_2}(x),
	\end{equation}
	where $Q_{a_1,a_2}(x)$ is a polynomial in $\F_{q^6}[x]$ whose coefficients are polynomials in $a_1$ and $a_2$.\\
	By a straightforward estimate, we note  that the $\deg(Q_{a_1,a_2}(x))$ is at most $4q^5+2q^4$. 
Since $q \geq 5$, $\deg(Q_{a_1,a_2}(x))$ is less than $q^6$. Therefore $Q_{a_1,a_2}(x)$ is the null polynomial. Consider the coefficient  $$a_1^{1 + q + q^2} a_2^{1 + q + 2 q^2 + q^4} (a_1^{q^3 + q^4} - a_2^{q^3})(a_1^{q^3 + q^4} + a_2^{q^3})$$
of the term $x^{q^3 + 2 q^4 + 3 q^5}$ of $Q_{a_1,a_2}(x)$ , it is zero if and only if either $a_2=a^{q+1}_1$ or $a_2=-a^{q+1}_1$.
In both cases, since up to the sign the coefficient of the term $x^{
	q + q^2 + q^4 + 3 q^5}$ is $a_1^{1 + q + q^2 + 2 q^3 + 
2 q^4}(\N_{q^6/q}(a_1)-1)^2$ and it vanishes, we get $\N_{q^6/q}(a_1)=1$. 
Therefore, putting $a_1=\lambda^{q-1}$, we obtain $\lambda U_{\varphi}=U_{2\psi_6}$ if $a_2=a^{q+1}_1$ and $\lambda U_{\varphi}=U_{2\psi^{(5)}_{6}}$ if $a_2=-a^{q+1}_1$. 
Hence the $\mathcal{Z}(\GaL)$-class of $L_{\psi_{6}}$ is two and, by Proposition $\ref{GAL_psi}$,  it is simple.
\end{proof}

\begin{proposition}\label{p:class2}
The $\mathcal{Z}(\GaL)$-class of $L_{\psi^{(k)}_{8}}$, $k=1,3$, is two. 
Moreover, $L_{\psi^{(k)}_{8}}$ is  a  simple linear set.
\end{proposition}

\begin{proof} First we prove the statement for $k=1$.
Since $\psi_{8}(x)$ and $\psi_8^{(7)}(x)$ define the same linear  set, we know $L_{2\psi_{8}}=L_{2\psi^{(7)}_{8}}$. Suppose $L_\varphi=L_{2\psi_8(x)}$ for some $\varphi(x)=\sum^7_{i=0}a_{i}x^{q^i} \in \F_{q^8}[x]$. We show that there exists $\lambda \in \F^*_{q^8}$ such that either $\lambda U_{\varphi}=U_{2\psi_{8}}$ or  $\lambda U_{\varphi}=U_{2\psi^{(7)}_{8}}$.\\
By \eqref{(1)}, \eqref{(2)} and \eqref{(3)} in Lemma \ref{Lemma3.6}, one obtains that 
\begin{equation}\label{condition2}
a_0=a_2=a_4=a_6=0,\quad a_1a^q_7=1 \quad \text{and} \quad  a_3a_5^{q^3}=-1.
\end{equation}
By Proposition \ref{Prop2.3}, the Dickson matrix associated with the $q$-polynomial
	\begin{equation}
	F(Y)=\varphi(x)Y-2\psi_8(Y)x
	\end{equation}
	has zero determinant $D_{F(Y)}(x)$ for each $x \in \F_{q^8}$, where, by \eqref{condition2}, $$\varphi(x)=a_1x^q+a_3x^{q^3}-a^{-q^5}_3x^{q^5}+a^{-q^{7}}_1x^{q^7}.$$
	Direct computation shows that 
	\begin{equation}
D_{F(Y)}(x)=\frac{1}{\N_{q^8/q}(a_1a_3)}Q_{a_1,a_3}(x),
	\end{equation}
	where $Q_{a_1,a_3}(x)$ is a polynomial in $\F_{q^8}[x]$ whose coefficients are polynomials in $a_1$ and $a_3$.
	By a straightforward estimate, we note  that the $\deg(Q_{a_1,a_3}(x))$ is at most $4(q^6+q^7)$. 
	\begin{itemize}
		\item [-] Case 1, $q \geq 5$. In this case  $4(q^6+q^7)$ is less than $q^8$. Therefore $D_{F(Y)}(x)$ is the null polynomial. Consider then the coefficient  $x^{q^4 + q^5 + 3 q^6 + 3 q^7}$ of $Q_{a_1,a_2}(x)$, it is zero if and only if either $a_3=a^{q^2+q+1}_1$ or $a_3=-a^{q^2+q+1}_1$.\\
	In both cases, then, since up to the sign the coefficient $$a_1^{2 + q^3 + 2 q^4 + 2 q^5 + 3 q^6 + 4 q^7}(\N_{q^8/q}(a_1)-1)^2$$
	 of term $x^{3 + 3 q + q^2 + q^5}$ is zero,  $\N_{q^8/q}(a_1)=1$ follows.\\
Therefore, putting $a_1=\lambda^{q-1}$, we obtain $\lambda U_{\varphi}=U_{2\psi_8}$ if $a_3=a^{q^2+q+1}_1$ and $\lambda U_{\varphi}=U_{2\psi^{(7)}_{8}}$ if $a_3=-a^{q^2+q+1}_1$. \\

	\item [-] Case $q=3$. 	
 Reduce the polynomial $Q_{a_1,a_3}(x)$ in \eqref{DF(Y)} modulo $(x^{q^8}-x)$, then  one gets that the coefficient of $x^{48}$ is $a_1^{5480} a_3^{4248} - a_1^{5454} a_3^{4250}$. Then, either $a_3=a_1^{q^2+q+1}$ or $a_3=-a_1^{q^2+q+1}$. In both cases, since $Q_{a_1,a_2}(x)$ ($\mod x^{q^8}-x$) has to be the null polynomial, up to sign the coefficient of term $x^{2439}$ is $a^{2860}(\N_{q^8}(a_1)-1)^2$, whence $\N_{q^8/q}(a_1)=1$.
Therefore, putting $a_1=\lambda^{q-1}$, we obtain $\lambda U_{\varphi}=U_{2\psi_8}$ if $a_3=a^{q^2+q+1}_1$ and $\lambda U_{\varphi}=U_{2\psi^{(7)}_{8}}$ if $a_3=-a^{q^2+q+1}_1$.
\end{itemize}

The computations for $k=3$ are similar and we omit to report them.
Hence the $\mathcal{Z}(\GaL)$-class of $L_{\psi^{(k)}_{8}}$ is two for $k=1,3$ and, 
by Proposition $\ref{GAL_psi}$, such linear set is simple.
\end{proof}

\begin{corollary}
	The linear sets $L_{\psi_{8}}$ and  $L_{\psi^{(3)}_{8}}$ are not $\PGaL$-equivalent.
\end{corollary}

\begin{proposition}
    Let $q\equiv1\pmod 4$.
	The $\mathcal{Z}(\GaL)$-class of $L_{\psi_{10}}$ is two. 
	Moreover, $L_{\psi_{10}}$ is a  simple linear set.
\end{proposition}
\begin{proof}
Like in the previous propositions, \eqref{(1)}, \eqref{(2)} and \eqref{(3)} in \ref{Lemma3.6} imply that if 
$L_\varphi=L_{2\psi_{10}(x)}$ for some $\varphi(x)=\sum^9_{i=0}a_{i}x^{q^i} \in \F_{q^{10}}[x]$,
then 
\begin{equation}\label{e:n10k1}
a_0=a_2=a_3=a_5=a_7=a_8=0,
\end{equation}
and $a_9=a_1^{-q^9}$, $a_6=-a_4^{-q^6}$.
The determinant $D_{F(Y)}(x)$ of the Dickson matrix associated with 
$\varphi(x)Y-(Y^q+Y^{q^4}-Y^{q^6}+Y^{q^9})x$
has degree at most $4q^9+4q^8+2q^7< q^{10}$, so it vanishes.
The coefficient of $x^{3+3q+2q^2+q^3+q^4}$ in $D_{F(Y)}(x)$ is
\[
a_1^{w_1}a_4^{w_2}\left(a_1^{2(1+q+q^2+q^3)}-a_4^2\right)
\]
for some $w_1,w_2\in\mathbb Z$, implying $a_4=\pm a_1^{1+q+q^2+q^3}$.
In both cases, by substituting such expressions of $a_4$ in $D_{F(Y)}(x)$, the coefficient of
$x^{q^3 + 3q^7 + 3q^8 + 3q^9}$ is
\[
\pm a_1^w\left(1-\N_{q^{10}/q}(a_1)\right)^2
\]
for some $w\in\mathbb Z$, whence $\N_{q^{10}/q}(a_1)=1$.
The proof can now be completed as in the Propositions  \ref{p:class1} and \ref{p:class2}.
\end{proof}

\begin{corollary}
The linear set $L_{\psi_{10}}$ is a new maximum scattered $\Fq$-linear set in $\PG(1,q^{10})$
($q\equiv1\pmod4$).
\end{corollary}
\begin{proof}
Since $L_{\psi_{10}}$ is a simple linear set, it is enough to check that $U_{\psi_{10}}$
does not belong to the $\GaL$-orbit of some $U_{s,\delta}^{2,10}$.
This will proved in Proposition \ref{p:s4} in a more general result.
\end{proof}

\begin{remark}
For $n=10$ and $k=3$, the equations in  Lemma \ref{Lemma3.6} does not imply
that six coefficients of $\varphi(x)$ are equal to zero,  like in \eqref{e:n10k1}.
This adds complexity to the computations.
\end{remark}

It is not known to the authors of this paper whether $L_{\psi_n^{(k)}}$ is a new linear set for $n>10$ and $\gcd(n,k)=1$
or $n=10$ and $k=3$. Indeed, it would be necessary to show that $L_{\psi^{(k)}_{n}}$ is not a  linear set of LP-type. Furthermore, the techniques used so far do not seem to be within easy reach when solving the issue of $\mathcal{Z}(\GaL)$- and $\GaL$-class of $L_{\psi_n^{(k)}}$ for $n>10$ and $\gcd(n,k)=1$ or $n=10$ and $k=3$ .

The  following result describes the intersection of $L_{\psi_n^{(k)}}$ with a special Baer subline.

\begin{proposition}
    Assume that $\psi_n^{(k)}$ is a scattered $q$-polynomial, $1\le k<t$.
	Let $\Sigma\cong\PG(1,q^t)$ be the subline of $\PG(1,q^n)$ consisting of the points represented by nonzero
	pairs in $\F_{q^t}^2$. Then $\Sigma\cap L_{\psi_n^{(k)}}$ is partitioned into two  
	$\Fq$-linear sets of pseudoregulus type
	of $\PG(1,q^t)$.
\end{proposition}

\begin{proof}
	The points of an $\Fq$-linear set of pseudoregulus type contained in $L_{\psi_n^{(k)}}$ are of type 
	$\la(h,h^{q^{t-k}})\ra_{\Fqn}$ for 
	$h\in\F_{q^t}^*$.
    Let $\xi\in\Fqt$ be such that $\N_{q^t/q}(\xi)=-1$.
    The map $\la (a,b) \ra_{\Fqn} \mapsto 	\la (a,\xi b) \ra_{\Fqn}$ induces a projectivity
    of $\Sigma$ mapping any of the $(q^t-1)/(q-1)$ points of type $\la(r,r^{q^k})\ra_{\Fqn}$, 
    $r\in W\setminus\{0\}$
    in a point having non-homogenous coordinate, say $\eta_r$, satisfying $\N_{q^t/q}(\eta_r)=1$.
    Therefore $M=\{\la(r,r^{q^k})\ra_{\Fqn}\colon r\in W\setminus\{0\}\}$ is a further linear set of 
    pseudoregulus type contained in $\Sigma$.

Next, let $P=\la (u,v)\ra_{\Fqn}$ be a point in $\Sigma$, with $u,v \in \Fqt$. 
	Then, $P$ belongs to $L_{\psi_n^{(k)}}$ if and only if there exists $\lambda, x \in \Fqn^*$ such that $x=\lambda u$ and $\psi^{(k)}(x)=\lambda v$. This is equivalent to
	$$\frac{\psi^{(k)}(x)}{x}=\frac{v}{u},$$
	whence $\psi^{(k)}(x)/x \in \Fqt$.
	Equivalently,
\[ 
\left( \frac{\psi^{(k)}(x)}{x} \right)^{q^t}= \frac{\psi^{(k)}(x)}{x}, \]
that can be reformulated in 
		\begin{equation}\label{alpha*beta}
	\alpha^{(k)}(x)\beta(x)^{q^{2t-1}}=\beta^{(k)}(x)\alpha(x)^q.
	\end{equation}
	Clearly the equation is satisfied by all $x$ either in $\Fqt$ or in $ W$.
	Now suppose that $x \in \Fqn \setminus (\Fqt \cup W)$. 
Then there exist $x_1 \in \Fqt$ and $x_2 \in W$, both nonzero, such that $x=x_1+x_2$.
	Next, \eqref{alpha*beta} implies
	$$\alpha^{(k)}(x_1)\beta(x_2)^{q^{2t-1}}=\beta^{(k)}(x_2)\alpha(x_1)^q.$$
	This is equivalent to
	$x_2^{q^k-1}=x_1^{q^{k(t-1)}-1}$; therefore, there exists $\mu \in \Fq$ such that $x_2=\mu x_1^{q^{k(t-2)}+\ldots+1}$, a contradiction.
\end{proof}

\section{New MRD-codes}\label{s:newmrd}


As recalled before, in  \cite[Section 5]{Sh} Sheekey explicated a link between  maximum scattered $\Fq$-linear sets of 
$\PG(1, q^n)$ and $\Fq$-linear MRD-codes with minimum distance $d=n-1$.
We briefly describe such relationship.
After fixing an $\Fq$-basis for $\Fqn$, we can define an isomorphism between the rings $\End_{\Fq}(\Fqn)$ and $\Fq^{n \times n}$ and then any RD-code can be seen as a subset of linearized polynomials over $\Fqn$.
Next, let $U_f = \{(x, f(x)): x \in \Fqn \}$ be an $\Fq$-subspace of $\Fqn \times \Fqn$, where $f(x)$  is  a $q$-polynomial over $\Fqn$.
The set
\begin{equation}\label{e:cf}
\cC_f = \{af(x) + bx: a, b \in \Fqn  \}=\la x, f(x) \ra_{q^n }
\end{equation}
corresponds to a subset of square matrices of order $n$ over $\Fq$ and hence to a rank distance code.
In particular, the following result holds:

\begin{theorem}\cite{Sh}
Let $f(x)$  be a linearized polynomial with $\deg_q(f) \leq n-1$. Then $\cC_f$ is an $\Fq$-linear MRD-code 
with parameters $(n, n, q; n - 1)$ if and only if $U_f$ is a maximum scattered $\Fq$-subspace
of $\Fqn \times \Fqn$, i.e., $f$ is a scattered $q$-polynomial.
\end{theorem} 
Moreover, in \cite{CMPZ}, the authors prove that $\cC_f$ is $\Fqn$-linear on the left, i.e.\ 
$L(\mathcal C_f ) \simeq \Fqn$, and any MRD-code with parameters $(n, n, q; n - 1)$ with left idealiser isomorphic to $\Fqn$  is equivalent to $\cC_f$, for some scattered $q$-polynomial $f(x)$, in \cite[Proposition 6.1]{CMPZ}. 
Finally, we recall the following result concerning the equivalence.
\begin{theorem}\cite{Sh}\label{t:citato}
	If $\cC_f$ and $\cC_g$ are two MRD-codes arising from
	maximum scattered subspaces $U_f$ and $U_g$ of $\Fqn \times \Fqn$ , then $\cC_f$ and $\cC_g$ are
	equivalent if and only if $U_f$ and $U_g$ are $\GaL(2, q^n)$-equivalent.
\end{theorem}

Therefore, if $\cC_f$ and $\cC_g$ are equivalent, one gets that the associated linear sets $L_f$
and $L_g$ are $\PGaL(2, q^n)$-equivalent. The converse statement does not hold in general, see \cite[Section 4.1]{OlgaFer}. By the results in Section \ref{newLS}, we have that 
\begin{itemize}
	\item [-] $U_{\psi_n^{(k)}}$ and $U^{1,n}_s$,
	\item [-] $U_{\psi_8^{(k)}}$ and $U^{2,n}_{s,\delta}$,
	\item [-] $U_{\psi_8^{(k)}}$ and $U^{3,8}_{s,\delta}$
\end{itemize}
give rise to pairwise inequivalent MRD-codes for any compatible $k$.
Then, to conclude the equivalence issue, we show the following
\begin{proposition}\label{p:s4}
	Let $t\ge3$ and assume that $\psi_n^{(k)}(x)$ is scattered, $1\le k<n$.
	Then the $\Fq$-subspaces ${U_{\psi_n^{(k)}}}$ and $U^{2,n}_{s,\delta}$ are not equivalent under the action of $\GaL(2,q^n)$. 
\end{proposition}

\begin{proof}
	Suppose that $U_{\psi_n^{(k)}}$ and $U^{2,n}_{s,\delta}$ are $\GaL(2,q^n)$-equivalent. 
	This is equivalent to suppose that $U_{\psi_n^{(k)}}$ and $U^{2,n}_{s,\delta}$ are 
	$\GL(2,q^n)$-equivalent.
	Furthermore, since, by Proposition \ref{GAL_psi}, $U_{\psi_n^{(k)}}$ and $U^{2,n}_{s,\delta}$ are 
	$\GaL(2,q^n)$-equivalent if and only if $U_{\psi_n^{(n-k)}}$ and $U^{2,n}_{s,\delta}$ are 
	$\GaL(2,q^n)$-equivalent, we may suppose $k <t$.
	Then, let $\begin{pmatrix}a&b\\ c&d\end{pmatrix}$ be an invertible matrix in $\F_{q^{n}}^{2\times2}$
	such that for any $x\in\F_{q^{n}}$ there exists $z\in\F_{q^{n}}$ such  that
	\[
	\begin{pmatrix}a&b\\ c&d\end{pmatrix}\begin{pmatrix}x\\ \psi_n^{(k)}(x)\end{pmatrix}=
	\begin{pmatrix}z\\ \delta z^{q^s}+z^{q^{n-s}}\end{pmatrix}.
	\]
	In particular, one obtains that for any $x\in\Fqt$
	\[
	cx+dx^{q^{k(t-1)}}=\delta \bigr (ax+bx^{q^{k(t-1)}} \bigr )^{q^s}+ \bigl (ax+bx^{q^{k(t-1)}} \bigr )^{q^{n-s}}.
	\]
	That is, any $x\in\Fqt$ is a root of the polynomial
	\begin{equation}\label{e:monomi}
	cx+dx^{q^{t-k}}-\delta a^{q^s}x^{q^s}-\delta b^{q^s}x^{q^{n-k+s}}-a^{q^{n-s}}x^{q^{n-s}}-b^{q^{n-s}}x^{q^{n-k-s}}.
	\end{equation}
	\begin{itemize}
		\item [$\circ$] \textsl{Case 1: $s<t$.}
		Then the polynomial in \eqref{e:monomi} becomes
		\begin{equation}\label{e:monomi-case1}
		cx+dx^{q^{t-k}}-\delta a^{q^s}x^{q^s}-\delta b^{q^s}x^{q^{e_1}}-a^{q^{n-s}}x^{q^{t-s}}-b^{q^{n-s}}x^{q^{e_2}},
		\end{equation}
		where $e_1$ and $e_2$ are the remainders of the divisions of $n-k+s$ and $n-k-s$ by $t$, respectively, and this polynomial
		is the null one.\\
		Call $M_1$, $M_2$, $\ldots$, $M_6$ the monomials in \eqref{e:monomi-case1}.	
		\begin{itemize}
			\item  \textit{Case 1a.} $s=k$. Since $k$ and $t$ are relatively prime, the integers $k$, $t-k$ and $e_2$ are distinct. Then
			
			$M_1$, $M_4$ are of degree $q^0$;\\ $M_2$, $M_5$ are of degree $q^{t-k}$;\\
			$M_3$ is of degree $q^{k}$;\\ 
			$M_6$ is of degree $q^{e_2}$.\\
			This implies $a=b=0$, a contradiction.
			\item \textit{Case 1b.} $s=t-k$.  Since $k$ and $t$ are relatively prime, the integers $k$, $e_1$  and $t-k$ are distinct.
			Then
			
			$M_1$, $M_6$ are of degree $q^0$;\\
			$M_2$, $M_3$ are of degree $q^{t-k}$;\\
			$M_4$ is of degree $q^{e_1}$;\\ 
			$M_5$ is of degree $q^{e_2}$.\\
			This implies $a=b=0$, a contradiction.
			\item \textit{Case 1c.} $s \neq k$ and $k+ s \neq t$. 
			In this case $M_1$ is the unique monomial of degree $q^0$, whence $c=0$;
			$M_1$ is the unique monomial of degree $q^{t-k}$, whence $d=0$, a contradiction.
		\end{itemize}
		\item [$\circ$] \textsl{Case 2: $s>t$.}
		Then one may suppose that $s=t+r$ with $r<t$. Then the polynomial in \eqref{e:monomi} becomes
		\begin{equation}\label{e:monomi-case2}
		cx+dx^{q^{t-k}}-\delta a^{q^s}x^{q^r}-\delta b^{q^s}x^{q^{d_1}}-a^{q^{n-s}}x^{q^{t-r}}-b^{q^{n-s}}x^{q^{d_2}},
		\end{equation}
		where $d_1$ and $d_2$ are the remainders of the divisions of $n-k+r$ and $n-k-r$ by $t$, respectively, and this polynomial
		is the null one.\\
		As before, call $N_1$, $N_2$, $\ldots$, $N_6$ the monomials in \eqref{e:monomi-case2}.
		Proceeding as in the previous case a contradiction is obtained.\qedhere
	\end{itemize} \end{proof}


In view of Theorem \ref{t:citato},
the following result summarizes Propositions \ref{GAL_psi}, \ref{p:s1}, \ref{p:s2}, \ref{p:s3} and \ref{p:s4}.

\begin{theorem}\label{t:finale}
  Let $t\ge3$ and $q$ odd if $t$ is even, or $q\equiv1\pmod4$ if $t$ is odd.
  Furthermore, let $1\le k<t$ be such that $\gcd(k,2t)=1$.
  Then the code $\cC_{\psi_{2t}^{(k)}}$ (cf.\ \eqref{e:cf} \eqref{e:kfold}) is an MRD-code with parameters
  $(2t,2t,q;2t-1)$ not equivalent to any previously known MRD-code.
  The $\varphi(2t)/2$ \footnote{Here $\varphi$ is Euler's totient function.} codes obtained in this way are distinct up to
  equivalence.
\end{theorem}

Giovanni Longobardi and Corrado Zanella\\
Dipartimento di Tecnica e Gestione dei Sistemi Industriali\\
Universit\`a degli Studi di Padova\\
Stradella S. Nicola, 3\\
36100 Vicenza VI\\
Italy\\
\emph{\{giovanni.longobardi,corrado.zanella\}@unipd.it}

\end{document}